\documentclass[a4paper,11pt]{article}

\usepackage[english]{babel}
\usepackage{color}
\usepackage{amsmath,amssymb,amsthm}
\usepackage{ stmaryrd }
\usepackage{latexsym}
\usepackage{graphicx,epsfig}
\usepackage{float,enumerate}
\usepackage[left=25mm, right=25mm,
            top=30mm, bottom=30mm]{geometry}
\newtheorem{theo}{Theorem}

\newtheorem{lem}{Lemma}

\newtheorem{remark}{Remark}

\newcommand{\bbZ}{\mathbb{Z}}
\newcommand{\set}[1]{\left\{#1\right\}}
\newcommand{\card}{\mathrm{card}}
\newcommand{\eps}{\varepsilon}

\newcommand{\N}{\mathbb{N}}
\newcommand{\Z}{\mathbb{Z}}

\def\rvsv{Rolla-Sidoravicius-Surgailis-Vares }

\title{Discrete Hammersley's lines with sources and sinks }
\author{A-L. Basdevant\footnote{Laboratoire
Modal'X, Universit\'e Paris Ouest, France. E-mail:
anne-laure.basdevant@u-paris10.fr}, N. Enriquez\footnote{Laboratoire
Modal'X, Universit\'e Paris Ouest, France. E-mail:
nenriquez@u-paris10.fr}, L. Gerin\footnote{CMAP, Ecole Polytechnique, France. E-mail:
gerin@cmap.polytechnique.fr}, J-B. Gou\'er\'e\footnote{MAPMO,
Universit\'e d'Orl\'eans, France. E-mail:
jean-baptiste.gouere@univ-orleans.fr}}

\begin{document}

\maketitle
\begin{abstract}
We consider stationary versions of two discrete variants of Hammersley's process in a finite box.
This allows us to recover in a unified and simple way the laws of large numbers proved by T. Sepp\"al\"ainen for two generalized Ulam's problems. As a by-product we obtain an elementary solution for the original Ulam problem.

We also prove that for the first process defined on $\Z$, Bernoulli product measures are the only extremal and translation-invariant stationary measures.
\end{abstract}
 

\section{Introduction}

In a celebrated paper, J.M.Hammersley used Poissonization to attack the so-called \emph{Ulam problem} of the typical length $\ell(n)$ of the longest increasing subsequence of a uniform permutation of size $n$. Namely, he reduced this problem to finding  the greatest number of points of a Poisson point process inside a square, an increasing path can go through.  He proved (\cite{HammersleyHistorique}, Theorem 4) that $\ell(n)/\sqrt{n}$ converges in probability to some constant $c$, sometimes refered to as the \emph{Ulam constant}, and conjectured that $c=2$.

The proof of $c=2$ was achieved independently by Logan and Shepp and by Vershik and Kerov in 1977, using algebraic methods. Various authors were then interested in finding a more probabilistic proof of this result. First,  Aldous and Diaconis \cite{AldousDiaconis} gave one, using the properties of what they called  Hammersley's process, which was implicitly introduced in \cite{HammersleyHistorique} (p.358 and following).  Hammersley's process is continuous in time and space and Aldous and Diaconis studied its properties on the infinite line, in particular its stationary distributions. A few years later, Groeneboom \cite{Groeneboom} and  Cator and Groeneboom \cite{CatorGroeneboom}  studied   Hammersley's process on a quarter plane.  By adding what they called Poisson \emph{sinks} and \emph{sources} on the $x$ and $y$-axis, they  also found a stationary version of this process on the quarter plane. Using this point of view, they   were able to recover again  the value of $c$.

In this paper, we study two discrete variants of  Ulam's problem. Namely, for all $p$ in $[0,1]$, $n,m\ge 1$ ,   we replace the original Poisson point process by the following random set $\xi$ of integer points: each integer point of the rectangle $ [1,n]\times [1,m] $ is chosen independently with probability $p$. We are interested in the two following quantities.

\begin{itemize}
\item The length of the longest increasing subsequence through points of $\xi$:
$$
L_{(n,m)}^{(1)}=\max\set{L; \hbox{ there exists }(i_1,j_1),\dots,(i_L,j_L)\in \xi,i_1< \dots< i_L \text{ and } j_1<j_2<\dots <j_L}.
$$
\item The length of the longest non-decreasing subsequence through points of $\xi$:
$$
L_{(n,m)}^{(2)}=\max\set{L; \hbox{ there exists } (i_1,j_1),\dots,(i_L,j_L)\in \xi,i_1< \dots< i_L \text{ and } j_1\leq j_2\leq \dots \le j_L}.
$$
\end{itemize}

One aim of the present work is to recover by simple and unified probabilistic arguments  the first order asymptotics of $L_{(n,m)}^{(1)}$ and $L_{(n,m)}^{(2)}$  already obtained by T. Sepp\"al\"ainen in two independent papers,  for $L_{(n,m)}^{(1)}$ in \cite{Sepp} and for $L_{(n,m)}^{(2)}$ in \cite{Sepp2}. 
In the following result the variables ${\left(L_{(n,m)}^{(i)}\right)}_{(n,m)}$ are coupled in the obvious way. 
Moreover, for any $a,b>0$, $L_{(an,bm)}^{(1)}$ stands for $L_{(\lfloor an \rfloor, \lfloor bn \rfloor)}^{(i)}$.

\begin{theo}\label{TheoLLN}
For $a,b>0$, we have, when $n$ tends to infinity
\begin{eqnarray}\label{eq:LGNModele1}
\label{Eq:LimiteModele1rect}\frac{L_{(an,bn)}^{(1)}}{n}&\stackrel{{\text{a.s.}}}{\to}& 
\begin{cases}
\displaystyle{\frac{\sqrt{p}(2\sqrt{ab}-(a+b)\sqrt{p})}{1-p}}&\mbox{ if } p<\min \{a/b,b/a\},\\
\min\{a,b\}&\mbox{ otherwise},\\
\end{cases}\\
\label{Eq:LimiteModele2rect}\frac{L_{(an,bn)}^{(2)}}{n}&\stackrel{{\text{a.s.}}}{\to}&
\begin{cases}
2\sqrt{a b p(1-p)}+(a-b)p &\mbox{ if } p<a/(a+b),\\
a &\mbox{ otherwise.}
\end{cases}
\end{eqnarray}

\end{theo}

To prove this, Sepp\"al\"ainen associates to each  problem a particle system  on the infinite line  $\Z$
and checks that Bernoulli product measures are stationary. The position of particles at a given time is then characterized as the solution of a discrete optimization problem, and the a.s. limit of $L_{(an,bn)}^{(i)}/n$ is  identified using convex analysis arguments.
Note also that for the second question, Johansson used later (\cite{Johansson}, Th.5.3) another description of $L_{(an,bn)}^{(2)}/n$ and proved that the rescaled fluctuations  converge to the Tracy-Widom distribution.

In \cite{SidoBroken2}, \rvsv considered the closely related model of last-passage percolation with geometric weights.
Here is a quick description of the model.
Let $(\tilde \xi(i,j))_{1 \le i \le n, 1\le j \le m}$ be a family of independent geometric random variables with parameter $p$.
Set:
$$
L_{(n,m)}^{(3)}=\max\set{ \sum_t \tilde \xi (\gamma(t))}
$$
where the maximum is taken over all paths $\gamma$ from $(1,1)$ to $(n,m)$ with $(1,0)$ or $(0,1)$ steps.
With respect to $L_{(n,m)}^{(1)}$ or $L_{(n,m)}^{(2)}$, we have simply changed the weight of each point (from Bernoulli to geometric) and 
the set of paths considered.
\rvsv gave a new proof of the following result:
$$
\frac{L_{(n,n)}^{(3)}}{n}\xrightarrow[n\to +\infty]{\text{a.s.}} \frac{2\sqrt{p}}{1-\sqrt p}.
$$
(The original proof was given by Jockusch-Propp-Shor in \cite{ArcticCircle}.)

Our proof of Theorem 1 is inspired by the proof of \rvsv of the above result.
The strategy is to consider a particle system on a \emph{bounded domain} which turns out to coincide with the restriction of Sepp\"al\"ainen's particle system on $\Z$. This simplifies the definition of the process (especially, in the case of the second problem). Moreover,
 it turns out that a {\it local} balance property around a single site, see Lemma \ref{Lem:Burke1} and \ref{Lem:Burke2} below, is enough to check the stationarity of the process. 
Our proofs are essentially the same for both models. 
This kind of remarkable local balance property also occurs in last-passage percolation with geometric weights 
(see \rvsv, Section 3.1 in \cite{SidoBroken2} or Sepp\"al\"ainen, Lemma 2.3 in \cite{Sepp3}). 
The particle systems studied here provide two other examples where such a property holds.
Theorem 1 is then proven by investigating the behaviour of the models under the different stationary measures which have been exhibited.

A nice by-product of our proof is that we obtain non-asymptotic
estimates which provide an elementary proof of $c=2$ for the original
Ulam problem (see the discussion at the end of Section \ref{sectionLLN}).

In the last section, which can be read independently, we complete the study of the original \emph{infinite} discrete Hammersley process introduced by Sepp\"al\"ainen \cite{Sepp} for the Problem 1. We prove that the only extremal translation-invariant stationary measures of this process are the Bernoulli product measures.

\section{Discrete Hammersley's processes}

Like Hammersley did, we construct two sets of broken lines  whose cardinality is respectively the variable $L_{(n,m)}^{(1)}$ and $L_{(n,m)}^{(2)}$. 
For each case, we first introduce a partial order:

$$
\begin{array}{c c c}
\text{Case 1:}              & \phantom{1302}&   \text{Case 2:}\\
(x,y) \prec (x',y')\text{ iff }   &  & (x,y) \prec (x',y')\text{ iff }\\
x<x' \text{ and }y<y'          &   & x<x' \text{ and }y\leq y'
\end{array}
$$
 
Now, Hammersley's lines are  paths starting from the top side of the rectangle $[1,n]\times [1,m]$, ending at its right side and making only south or east steps. 
They are constructed recursively. The first line is the highest non-increasing path connecting the minimal points of $\xi$ for $\prec$. We withdraw these points from $\xi$ and connect the new minima to get the second line, and so on.
In the below picture, $n=m=8$ and  $L_{(n,m)}^{(1)}=4$, $L_{(n,m)}^{(2)}=5$ (crosses denote points of $\xi$, Hammersley's lines are in blue, one of the longest subsequences is shown in a red dashed line):
\begin{center}
\begin{tabular}{c c c}
\includegraphics[width=55mm]{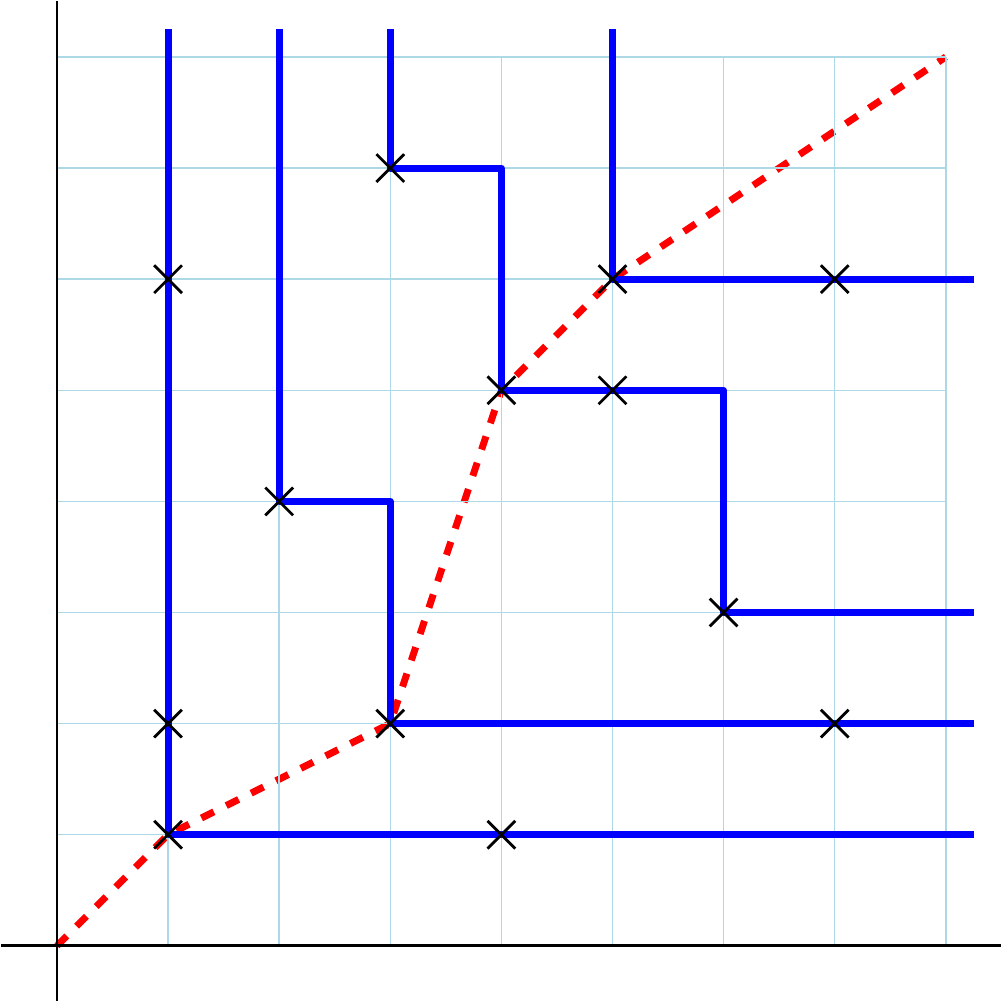} & \hspace{8mm} & \includegraphics[width=55mm]{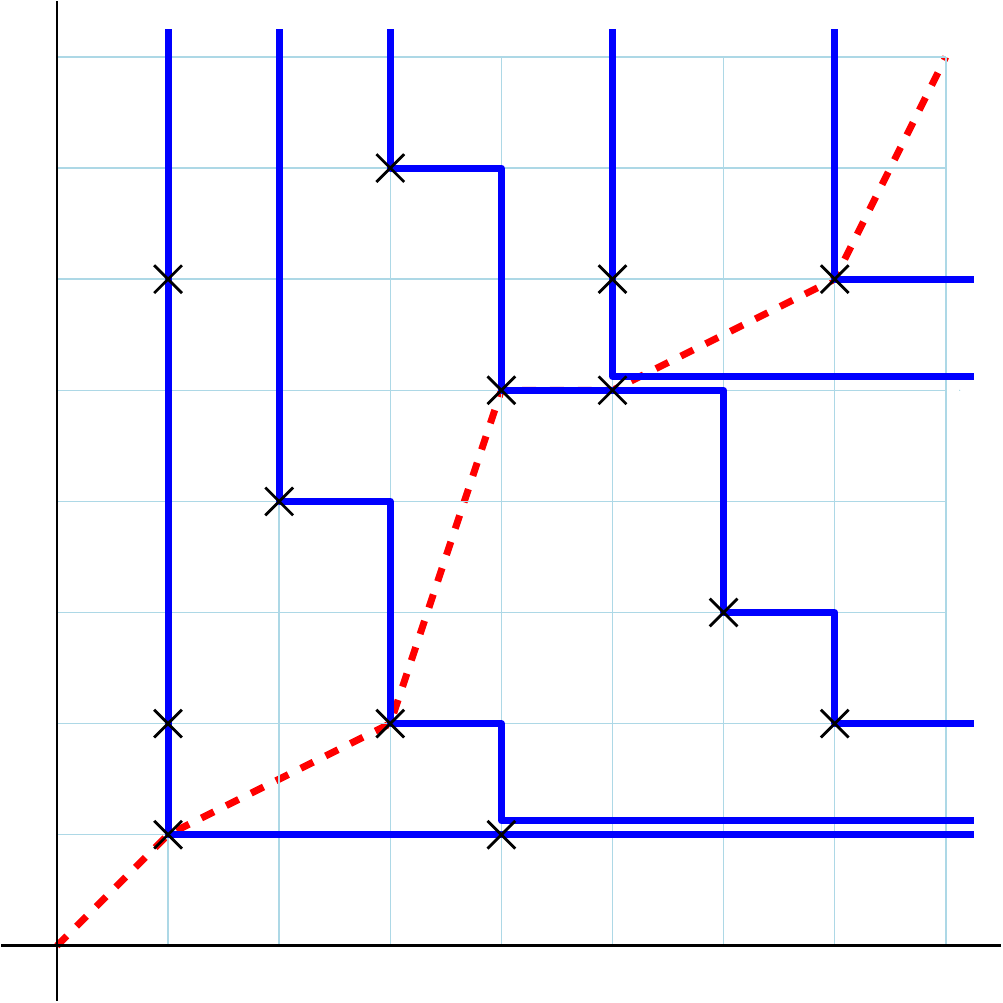}\\
Problem 1 & & Problem 2 
\end{tabular}

\end{center} We claim that
 the number  $H_{(n,m)}^{(i)}$ of Hammersley's lines is equal to $L_{(n,m)}^{(i)}$. Indeed,
note that, in both cases, $L_{(n,m)}^{(i)}$ is equal to the length of the longest increasing subsequence trough points of $\xi$ for the partial order  $\prec$. Each Hammersley's line connecting minimal points of a subset of $\xi$, the $L_{(n,m)}^{(i)}$ points
of a maximizing sequence necessarily lie on distinct Hammersley's lines
and therefore $H_{(n,m)}^{(i)}\geq L_{(n,m)}^{(i)}$. To prove the opposite inequality, we observe that Hammersley's lines
give a construction of an increasing sequence of points $a_1,\dots
a_{H_{(n,m)}^{(i)}}$ of $\xi$ in the following way. For $a_{H_{(n,m)}^{(i)}}$ we
take any point of $\xi$ which lies in the top right Hammersley's line.
Then, for each $\ell$, $a_{\ell -1}$ is taken as one of the points of
$\xi$ which lies in the $\ell-1$-th line and such that $a_{\ell -1}\prec a_{\ell} $. The existence of such $ a_{\ell -1}$ is granted. Indeed, if no such point existed,  $a_{\ell}$ would have belonged to the $(\ell-1)$-th line. This proves $L_{(n,m)}^{(i)}\geq
H_{(n,m)}^{(i)}$.

Let us note that the construction of lines above is consistent in $n$ and $m$, \emph{i.e.} for $n'<n$ and $m'<m$ the restriction of Hammersley's lines to the rectangle $[1,n']\times [1,m']$ coincides with Hammersley's lines constructed from the points of $\xi$ in this smaller rectangle.
In particular, for $n\ge 1$, we can define the two discrete time processes
 $(X_t^1)_{t\ge 0}:=(X^1_t(x), \ x\in \llbracket 1,n \rrbracket)_{t\ge 0}$ and $(X_t^2)_{t\ge 0}:=(X^2_t(x),\  x\in \llbracket 1,n \rrbracket)_{t\ge 0}$ defined by 
  \begin{equation}\label{eq:defX}
      X_t^i(x)=\left\{
        \begin{array}{ll}
          1 &\mbox{if there is a Hammersley's line on the vertical edge } \{(x,t),(x,t+1)\} \\
           0 & \mbox{otherwise.}
        \end{array}\right.
    \end{equation}
 Note that on each vertical edge there is at most one Hammersley's line. Moreover, in Model 1, there is also at most one Hammersley's line on any horizontal edge
 and Hammersley's lines do not intersect each others.
 We will say that there is a  \emph{particle} at time $t$ at position $x$ in Model $i$ if $X^i_t(x)=1$.  Hence, for both models we start with the initial condition $X^i_0 \equiv 0$ and  one can check that both processes are Markovian. Besides, since Hammersley's lines start all from the top side of the rectangle,
\begin{equation*}\label{e:salut}
\hbox{ for both models, the number of particles at time }t \hbox{ is equal to }L_{(n,t)}^{(i)}.
\end{equation*}

In both models, the dynamic of the particle system is quite simple. 
As the description of the dynamic is not needed in the proof of Thereom \ref{TheoLLN}, we do not give such a description here.
However, we refer the interested reader to Appendix \ref{s:leretour}.

\section{Sinks and sources}\label{sectionsinksources}
Since, in both models, the number of particles can only increase as time goes on,  both processes converge to their unique stationary measure  $X_\infty^i:\equiv 1$. In this section, we are going to modify a little bit the problems we are looking at, such that the particle systems associated to the new problems admit less trivial stationary measures. To do so, we must first introduce some notation.

From now on, for any integer  $x \ge 1$ and $t \ge 1$, we define the random variable $\xi_t(x)$ by 
\begin{equation} \label{e:tram}
      \xi_t(x)=\left\{
        \begin{array}{ll}
          1 &\mbox{if  }(x,t)\in \xi \\
           0 & \mbox{otherwise.}
        \end{array}\right.
    \end{equation}
 With this notation, we can write the quantities $L_{(n,m)}^{(i)}$ as 
\begin{equation}\label{eq:defmax}L_{(n,m)}^{(i)}=\max\set{ \sum_{s}\xi_{t(s)}(x(s))}
\end{equation}
where the maximum is taken over all the discrete paths $(x(\cdot),t(\cdot))$ in $[1,n]\times[1,m]$ starting from $(1,1)$ which make only steps in $\N \times \N$ for Model 1 and in $\N \times \Z_+$ for Model 2.
Given two sequences of integers  $(\xi_t(0))_{t\in  \llbracket 1,m \rrbracket}$  and  $(\xi_0(x))_{x\in \llbracket 1,n \rrbracket}$, we now consider 
the quantities 
\begin{equation}\label{eq:defmax2}\mathcal{L}_{(n,m)}^{(i)}=\max\set{ \sum_{s} \xi_{t(s)}(x(s))}
\end{equation}
where the maximum is taken over all the discrete paths $(x(\cdot),t(\cdot))$ in $[0,n]\times[0,m]$ starting from $(0,0)$ such that:
\begin{itemize}
 \item There exists an integer $u \in \Z_+$ such that: either they first make $u$ $(0,1)$ steps or they first make $u$ $(1,0)$ steps.
 \item Then, they only make steps in $\N \times \N$ for Model 1 and in $\N \times \Z_+$ for Model 2.
\end{itemize}
Using the same terminology as in \cite{CatorGroeneboom}, we say that there is $k$ \emph{sinks} (resp. \emph{sources}) at site $(0,t)$ (resp. $(x,0)$) if $\xi_t(0)=k$ (resp. $\xi_0(x)=k$). The sources will be drawn on the $x$-axis and the sinks on the $y$-axis.
Hence, in other words, $\mathcal{L}_{(n,m)}^{(i)}$ is equal to the maximum number of points an increasing (resp. non-decreasing) path can go trough when the path can pick either sources or either sinks before going into $\N \times \N$.
   The law of the sources and sinks will be specified later,  let us just say that the sources will take values in $\{0,1\}^n$ whereas the sinks will take values in $\{0,1\}^m$ for Model 1 and in $\Z_{+}^m$ for Model 2. 
   
\begin{figure}[h!]
\begin{center}
\begin{tabular}{c c c}
\includegraphics[width=7cm]{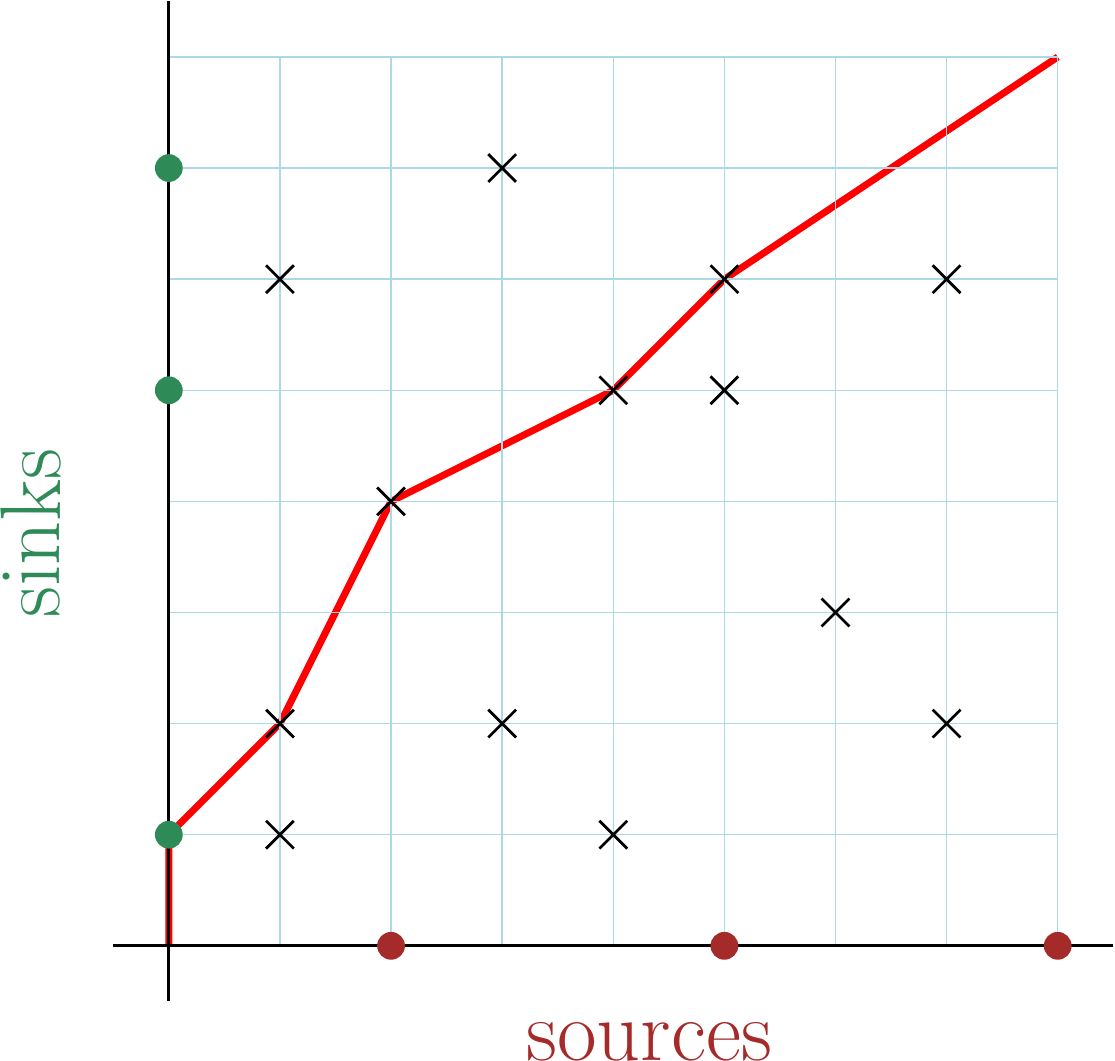} & \ \ \ \ & \includegraphics[width=7cm]{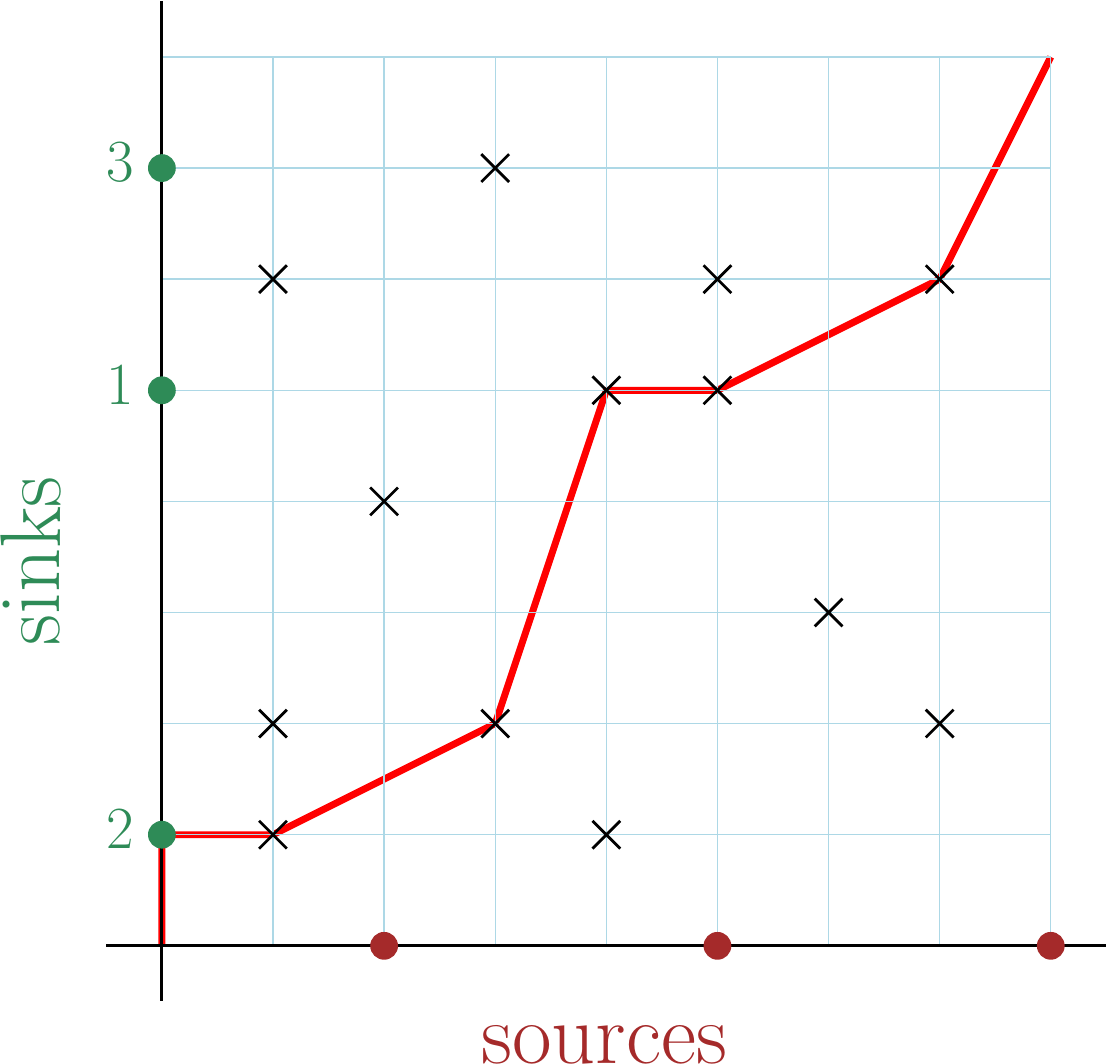}\\ 
Model 1, $\mathcal{L}_{(n,n)}^{(1)}=5$ & & Model 2, $\mathcal{L}_{(n,n)}^{(2)}=7$ 
\end{tabular}
\end{center}
\end{figure}

As before, we can construct a  set of broken lines whose cardinality is equal to $\mathcal{L}_{(n,m)}^{(i)}$. 
Now, Hammersley's lines are  paths starting from the top side or the left side of the rectangle $[0,n]\times [0,m]$, ending at its bottom side or its right side and making only south or east steps. 
They are still constructed recursively. Assume that the sink of minimal height is located at $(0,t_1)$ and the leftmost source at $(x_1,0)$. Then, the first line is the highest non-increasing path starting from $(0,t_1)$, ending at $(x_1,0)$ and connecting minimal points $(x,t)$ of $\xi$ for the partial order $\prec$ introduced in the previous section such that $(x_1,0)\nprec (x,t)$
and $(0, t_1)\nprec (x,t)$. If there is no sink (resp. source) at all, we do the same construction except that the line starts from the top of the box (resp. ends on the right of the box).
Then, we withdraw these points from $\xi$, one of the sink located at $(0,t_1)$ and the source located at $(x_1,0)$ and we do the same procedure with this new set of sources, sinks, and set of points  in $\llbracket 1, n\rrbracket\times \llbracket 1, m\rrbracket$ to obtain the second line.
 The algorithm goes on until there is no more points, sinks and sources.
 Using the same argument as in the previous section, one can easily prove that the number of lines obtained with this procedure is equal to $\mathcal{L}_{(n,m)}^{(i)}$.

\begin{figure}[h!]
\begin{center}
\begin{tabular}{c c c}
\includegraphics[width=7cm]{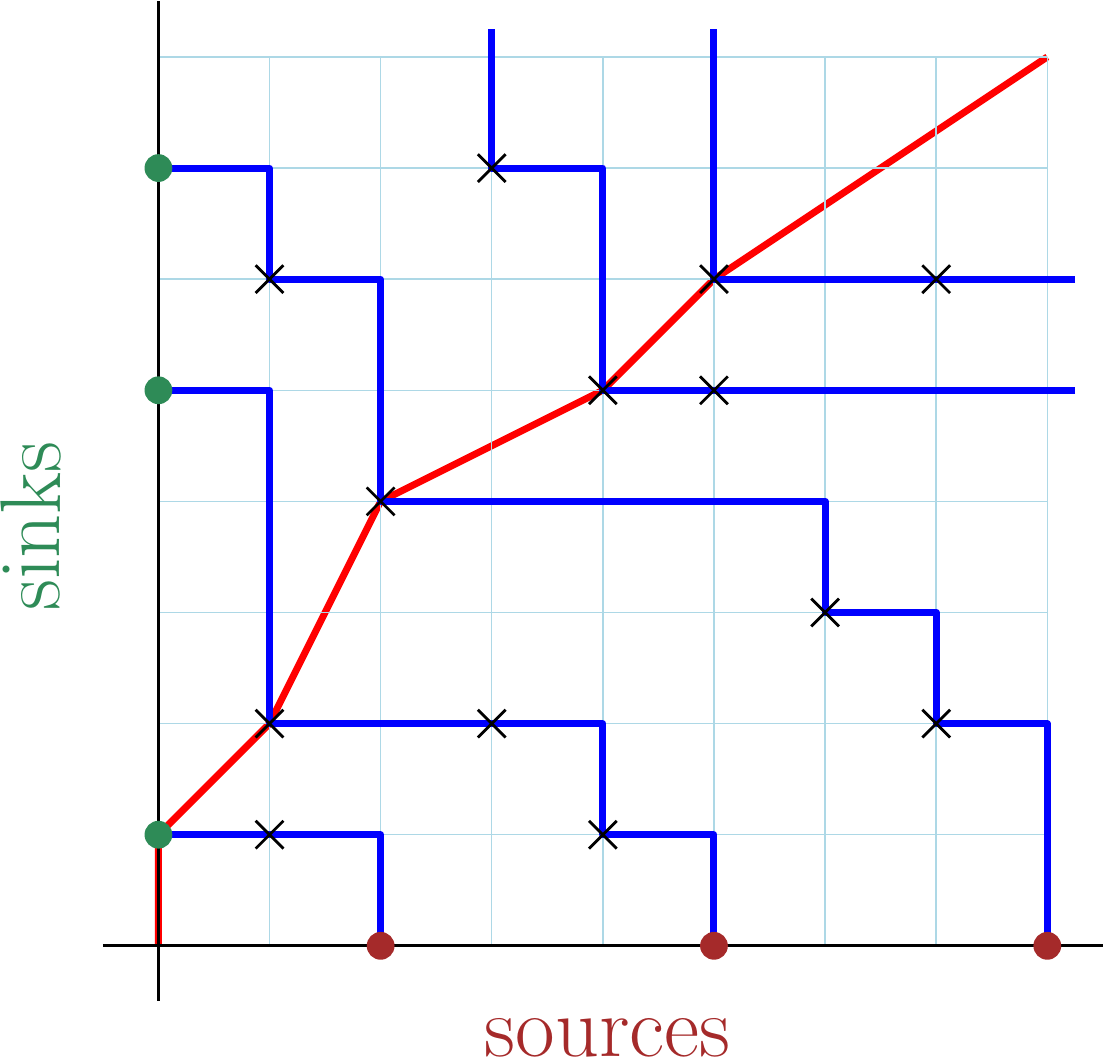} & \ \ \ \ & \includegraphics[width=7cm]{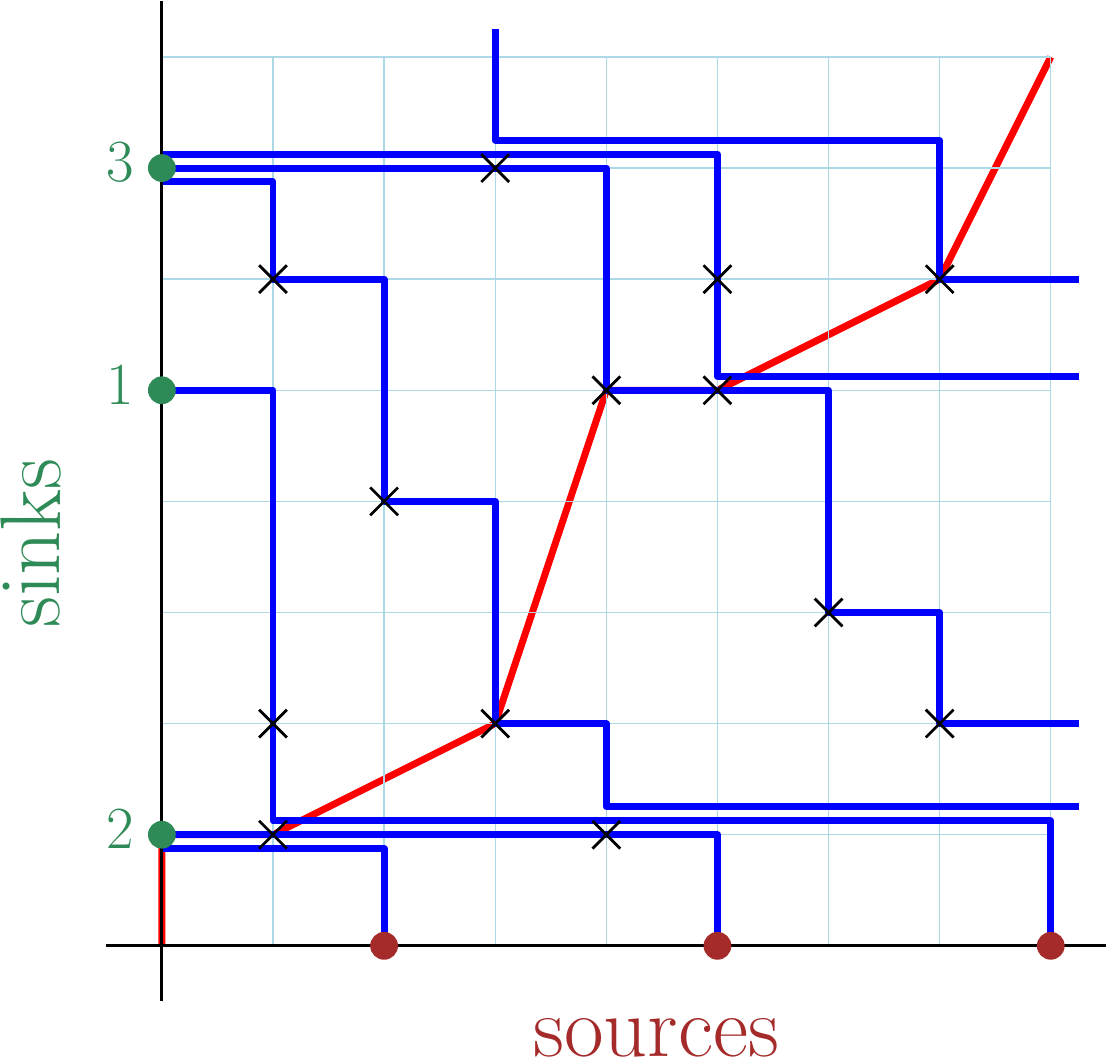}\\ 
Model 1, $\mathcal{L}_n^{(1)}=5$ & & Model 2, $\mathcal{L}_n^{(2)}=7$ 
\end{tabular}
\end{center}
\end{figure} 
 
 With a slight abuse of notation, we still denote 
 $(X_t^1)_{t\ge 0}:=(X^1_t(x), \ x\in \llbracket 1,n \rrbracket)_{t\ge 0}$ and $(X_t^2)_{t\ge 0}:=(X^2_t(x),\  x\in \llbracket 1,n \rrbracket)_{t\ge 0}$  the  two discrete time system of particles defined by 
  \begin{equation*}
      X_t^i(x)=\left\{
        \begin{array}{ll}
          1 &\mbox{if there is a Hammersley line (with sources and sinks) on the edge } \{(x,t),(x,t+1)\} \\
           0 & \mbox{otherwise.}
        \end{array}\right.
    \end{equation*}
Hence, the processes now start  from the configuration given by the sources \emph{i.e} $X^i_0(x)=\xi_0(x)$.

\bigskip

\begin{theo}[Stationarity of Hammersley's processes on a bounded interval, with sources and sinks]\label{Th:StatioSourcesPuits}\
\noindent For all $n$,
\begin{itemize}
\item {\bf Model 1.} For all $p,\alpha\in(0,1)$, the process $(X^1_t(x), x\in \llbracket 1,n \rrbracket)_{t\ge 0}$  is stationary if sources are i.i.d. $\mathrm{Ber}(\alpha)$ and sinks are i.i.d. $\mathrm{Ber}(\alpha^\star)$ (and sinks independent from sources) with
\begin{equation}\label{condmodel1}
\alpha^\star = \frac{p(1-\alpha)}{\alpha+p(1-\alpha)}.
\end{equation}
\item {\bf Model 2.} For all $p,\alpha\in(0,1)$ such that $\alpha> p$, the process $(X^2_t(x), x\in \llbracket 1,n \rrbracket)_{t\ge 0}$ is stationary if sources are i.i.d.  $\mathrm{Ber}(\alpha)$   and sinks are i.i.d. $(\mathrm{Geo}(\alpha^\star)-1)$ (i.e. $\mathbb{P}(\xi(0,t)=k)=\alpha^\star(1-\alpha^\star)^k$ for $k=0,1,\dots$) (and sinks independent from sources) with
\begin{equation}\label{condmodel2}
\alpha^\star=\frac{\alpha-p}{\alpha(1-p)}.
\end{equation}
(Note that $\alpha^\star \in (0,1)$ if $\alpha >p$.)
\end{itemize}
\end{theo}
\begin{remark} In \cite{Sepp} (resp. in \cite{Sepp2}), Sepp\"al\"ainen associates to Model 1 (resp. Model 2) a particle system defined on the whole line $\Z$. One can check that the restrictions of these systems to $\llbracket 1, n\rrbracket$ have respectively the same transition probabilities that our processes $X^i_t$, $i=1,2$ when sinks and sources are distributed as in Theorem \ref{Th:StatioSourcesPuits}.  Hence,
\begin{itemize}
\item For Model 1 (resp. Model 2), Theorem \ref{Th:StatioSourcesPuits} should be compared to  Lemma 2.1 in \cite{Sepp} (resp. Proposition 1 in \cite{Sepp2}) which states that Bernoulli product measures are stationary for the analogous model on the infinite line.
\item The condition $\alpha> p$ in Model 2 reminds the condition on the density of the initial configuration needed to define the process on $\Z$ in \cite{Sepp2}.
\item The theorem states stationarity in time (\emph{i.e.} from bottom to top in our figures). In fact, the proof also shows stationarity from left to right.
\end{itemize}
\end{remark}

Let us collect a few consequences for further use.
Let $\mathcal{L}_{(n,m)}^{(i),\alpha,\beta}$ be the number of Hammersley's lines in Model $i$ with parameters $\alpha$ (sources), $\beta$ (sinks)  in the box $[0,n]\times[0,m]$ and let  $\mathcal{T}_{(n,m)}^{(i),\alpha,\beta}$ be the number of Hammersley's lines that leave the same box from the top.
Then, for $i=1,2$ and every $\alpha,\beta\in [0,1]$
\begin{align}
\nonumber\mathcal{L}_{(n,m)}^{(i),0,0}& = \mathcal{T}_{(n,m)}^{(i),0,0} \stackrel{\text{d}}{=} L_{(n,m)}^{(i)}, \\
\label{Eq:DominationPuits2} \mathcal{L}_{(n,m)}^{(i),\alpha,\beta}&= \mathcal{T}_{(n,m)}^{(i),\alpha,\beta}+\#\set{\text{sinks between $1$ and $m$}}.
\end{align}
This relies on the fact that any lines has to exit by the top or by a sink.
Note that in the right-hand side of \eqref{Eq:DominationPuits2} the two terms are not independent. 
Moreover, using the interpretation of $L_{(n,m)}^{(i)}$ and  $\mathcal{L}_{(n,m)}^{(i),\alpha,\beta}$ as the length of a longest subsequence (\emph{cf}. \eqref{eq:defmax} and \eqref{eq:defmax2}), we  have
\begin{equation}
\label{Eq:DominationPuits1} L_{(n,m)}^{(i)} \le \mathcal{L}_{(n,m)}^{(i),\alpha,\beta}.
\end{equation}
Besides, $\mathcal{T}_{(n,m)}^{(i),\alpha,\beta}$ being by definition equal to the number of particles at time $m$, Theorem \ref{Th:StatioSourcesPuits} implies that
\begin{equation}\label{Eq:Binom}
\mathcal{T}_{(n,m)}^{(i),\alpha,\alpha^\star} \stackrel{\text{law}}{=} \mathrm{Binom}(n,\alpha).
\end{equation}

\begin{proof}
We aim to prove Theorem \ref{Th:StatioSourcesPuits}.
We first prove for both models a local balance property, in Lemmas 1 and 2. These lemmas are elementary but do not seem to be written elsewhere, and form the heart of our proof. 

\noindent{\bf Model 1.}\\
We first focus on what happens around a single point $(x,t)$. Recall that $\xi_t(x)$ is a Bernoulli$(p)$. Denote by $X$ (resp. $Y,X',Y'$) be the indicator that a Hammersley's line hits $(x,t)$ from the bottom (resp. left,top,right).

\begin{lem}[Local balance for Model 1]\label{Lem:Burke1} The random variables $(X',Y')$ are measurable with respect to $(X,Y,\xi_t(x))$ and we have
$$
\left((X,Y,\xi_t(x))\stackrel{\text{(d)}}{=} \mathrm{Ber}(\alpha)\otimes \mathrm{Ber}(\alpha^\star)\otimes \mathrm{Ber}(p)\right)
\Rightarrow
\left((X',Y')\stackrel{\text{(d)}}{=} \mathrm{Ber}(\alpha)\otimes \mathrm{Ber}(\alpha^\star)\right)
$$
where $\alpha^\star$ is as in \eqref{condmodel1}.
\end{lem}
\begin{center}
\includegraphics[height=32mm]{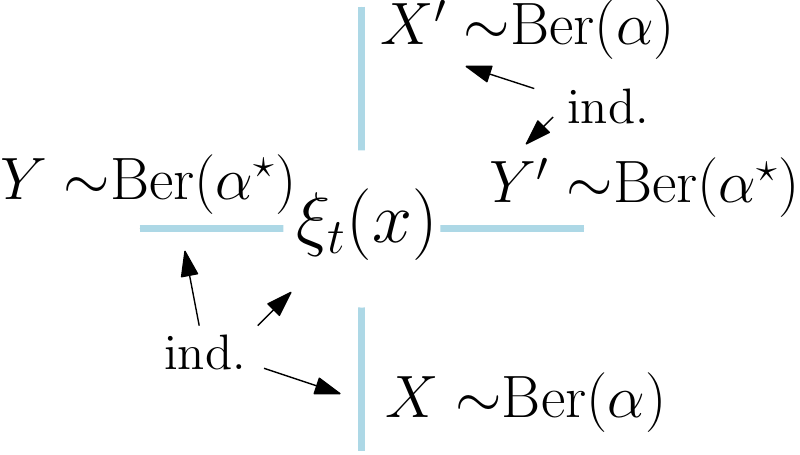}
\end{center}
\begin{proof}[Proof of Lemma \ref{Lem:Burke1}]Recalling that, in Model 1, Hammersley's lines are non increasing lines which do not touch each other and noticing that $X'=Y'=1$ iff $ X=Y=0$ and $\xi_t(x)=1$, we get that
\begin{equation*}
      (X',Y')=\left\{
        \begin{array}{ll}
          (X,Y) & \mbox{if\ }X\neq Y \\
          (1,1) & \mbox{if } X=Y=0 \mbox{ and } \xi_t(x)=1.\\
          (0,0)  & \mbox{otherwise.}
        \end{array}\right.
    \end{equation*}
 Hence, we see that  $(X',Y')\in \sigma(X,Y,\xi_t(x))$ and one can easily check that if $(X',Y')\sim \mathrm{Ber}(\alpha)\otimes \mathrm{Ber}(\alpha^\star)$ with $\alpha^\star$ is as in \eqref{condmodel1},  then $(X',Y')\sim \mathrm{Ber}(\alpha)\otimes \mathrm{Ber}(\alpha^\star)$.
\end{proof}
We now explain how Lemma \ref{Lem:Burke1} proves Theorem \ref{Th:StatioSourcesPuits} for Model 1.
For every $(x,1)$ define the corresponding $(X'(x),Y'(x))$ be the \emph{output} of $(x,1)$. Independence of sources and sinks and Lemma \ref{Lem:Burke1} ensure that the output of $(1,1)$ is independent of sources at $(2,0),(3,0),\dots$ and sinks at $(0,2),(0,3),\dots$. Then a simple induction on $x$  proves that the random variables $(X'(x),x\in \llbracket 1,n \rrbracket)$ are distributed as Ber$(\alpha)$.

\noindent{\bf Model 2.} The strategy is similar:
\begin{lem}[Local balance for Model 2]\label{Lem:Burke2}The random variables $(X',Y')$ are measurable with respect to $(X,Y,\xi_t(x))$ and we have
$$
\left((X,Y,\xi_t(x))\stackrel{\text{(d)}}{=} \mathrm{Ber}(\alpha)\otimes( \mathrm{Geo}(\alpha^\star)-1)\otimes \mathrm{Ber}(p)\right)
\Rightarrow
\left((X',Y')\stackrel{\text{(d)}}{=} \mathrm{Ber}(\alpha)\otimes (\mathrm{Geo}(\alpha^\star)-1)\right)
$$
where $\alpha^\star$ is as in \eqref{condmodel2}.
\end{lem}
\begin{center}
\includegraphics[height=32mm]{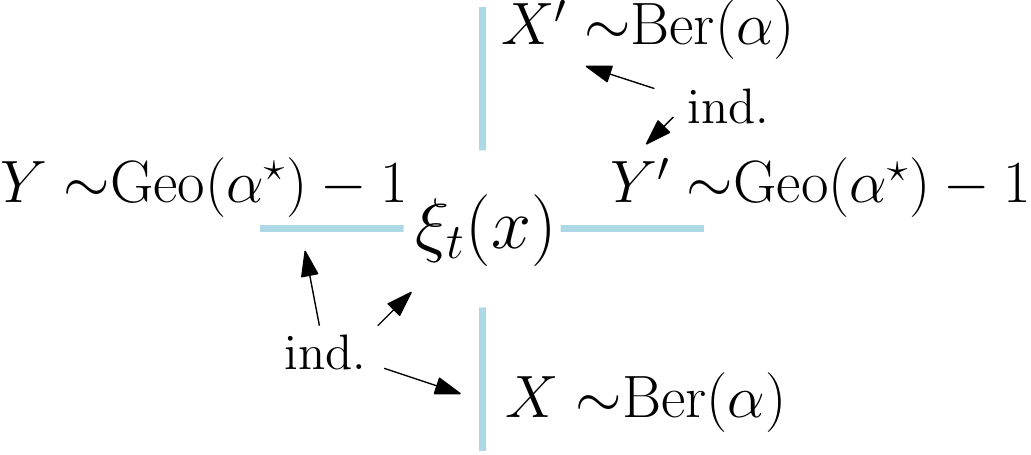}
\end{center}

\begin{proof}[Proof of Lemma \ref{Lem:Burke2}] Recall that $X,X'\in \{0,1\}$ whereas $Y,Y'\in \Z^+$. Besides, using that Hammersley's lines are non-increasing, the balance between incoming and outcoming lines at a site yields  $X'+Y=X+Y'$. Moreover, $X'=1$ iff ($\xi_t(x)=1$ or $(X,Y)=(1,0)$).

We observe that the different cases can be summed up in
\begin{equation*}
      (X',Y')=\left\{
        \begin{array}{ll}
          (1,0) & \mbox{if\ }(X,Y)=(1,0) \\
          (\xi_t(x),Y-X+\xi_t(x))& \mbox{otherwise.}
        \end{array}\right.
    \end{equation*}
One can easily check that this equality implies Lemma \ref{Lem:Burke2}.
For instance, for $k\ge 1$, $\mathbb{P}(X=1,Y=k)=\alpha\alpha^\star (1-\alpha^\star )^k$ should be equal to
\begin{align*}
\mathbb{P}(X'=1,Y'=k)&=\mathbb{P}(\xi_t(x)=1,X=0,Y=k-1)+ \mathbb{P}(\xi_t(x)=1,X=1,Y=k)\\
&= p(1-\alpha)\alpha^\star(1-\alpha^\star)^{k-1}+p\alpha\alpha^\star(1-\alpha^\star)^{k}\\
&=p(1-\alpha^\star)^{k-1}\alpha^\star\left[(1-\alpha)+\alpha(1-\alpha^\star)\right],
\end{align*}
which requires \eqref{condmodel2}.
\end{proof}

Theorem  \ref{Th:StatioSourcesPuits} for Model 2 then follows from Lemma \ref{Lem:Burke2} in the same way
as before Theorem  \ref{Th:StatioSourcesPuits} for Model 1  follows from Lemma \ref{Lem:Burke1}.
\end{proof}

\section{Law of Large Numbers in Hammersley's processes}\label{sectionLLN}
In this section, we explain how Theorem \ref{Th:StatioSourcesPuits}  on the stationary measure of the processes with sources and sinks implies Theorem \ref{TheoLLN}. 

\begin{proof}[Proof of Theorem \ref{TheoLLN}]

\subsubsection*{Trivial case.} 

We first  consider Model 1 in the case  $p \ge \min\{a/b,b/a\}$.
The upper bound is straightforward. 
The lower bound can be proven as follows.
Let us consider for example the case $p \ge ab^{-1}$.
We have to prove:
\begin{equation}\label{e:hunter}
\liminf_{n \to \infty} \frac{L_{(an,bn)}^{(1)}}{n} \ge a.
\end{equation}
One can build an increasing path as follows.
The first point $(1,y_1)$ of the path is the lowest point of $\xi$ having first coordinate equal to $1$.
The second point $(2,y_2)$ of the path is the lowest point of $\xi$ having first coordinate equal to $2$
and second coordinate strictly larger than $y_1$.
The other points are defined in a similar fashion.
As $pb \ge a$, a study of this path easily provides \eqref{e:hunter}.
The case $p \ge ba^{-1}$ can be handled in a similar way.

The same strategy treats the case $p\geq a/(a+b)$ in Model $2$. 

\subsubsection*{Non trivial case}

We only need to compute $\lim \frac{1}{n} \mathbb{E}[L_{(an,bn)}^{(i)}]$ since almost sure  convergence follows from superadditivity.

\noindent {\bf Model 1. Upper bound.} We assume that
$p\geq \min \{a/b,b/a\}$.  For any $\alpha \in (0,1)$, taking $\alpha^\star$ as in \eqref{condmodel1} and using    \eqref{Eq:DominationPuits2}, \eqref{Eq:DominationPuits1} and \eqref{Eq:Binom}, we get
\begin{eqnarray}
\frac{1}{n}\mathbb{E}[L_{(an,bn)}^{(1)}]&\leq&  \frac{1}{n}\mathbb{E}[\mathcal{L}_{(an,bn)}^{(1),\alpha,\alpha^\star}] \le  \frac{1}{n}\mathbb{E}[\mathcal{T}_{(an,bn)}^{(1),\alpha,\alpha^\star}] + \frac{1}{n}\mathbb{E}[\#\set{\text{sinks between $1$ and $bn$}}] \notag\\
&=& a\alpha + b\alpha^\star
= a\alpha +b\frac{p(1-\alpha)}{\alpha+p(1-\alpha)} =:\phi^{(1)}(a,b,\alpha). \label{eq:tgv}
\end{eqnarray}
The latter is minimized for
\begin{equation}\label{Eq:OptiRectangle}
\alpha(a,b):=\frac{\sqrt{p}\sqrt{\tfrac{b}{a}}-p}{1-p},\qquad \alpha^\star(a,b):=\frac{\sqrt{p}\sqrt{\tfrac{a}{b}}-p}{1-p},
\end{equation}
which both are in $(0,1)$ if $p<\min \{a/b,b/a\}$.
This yields
$$\frac{1}{n}\mathbb{E}[L_{(an,bn)}^{(1)}]\leq \phi^{(1)}(a,b,\alpha(a,b))= \frac{\sqrt{p}(2\sqrt{ab}-(a+b)\sqrt{p})}{1-p}.$$

\medskip

\noindent {\bf Model 2. Upper bound.}
We 	assume that  $p\geq a/(a+b)$. Using \eqref{Eq:Binom}   with $\alpha^\star=\frac{\alpha-p}{\alpha(1-p)}$ and that $\mathbb{E}[\mathrm{Geo}(\alpha^\star)-1]=\frac{1}{\alpha^\star}-1$, we get
\begin{eqnarray*}
\frac{1}{n}\mathbb{E}[L_{(an,bn)}^{(2)}]&\leq&  \frac{1}{n}\mathbb{E}[\mathcal{T}_{(an,bn)}^{(i),\alpha,\alpha^\star}] + \frac{1}{n}\mathbb{E}[\#\set{\text{sinks between $1$ and $bn$}}]\\
&=& a\alpha+ b(\frac{1}{\alpha^\star}-1)
= a\alpha +b\frac{p(1-\alpha)}{\alpha-p}.
\end{eqnarray*}
We minimize the latter by taking
$$
\alpha=p+\sqrt{\frac{b}{a}p(1-p)}.
$$
Note that this choice is allowed since for $p\in(0,a/(a+b))$ then $0<p<\alpha<1$.

\medskip

\noindent {\bf Model 1 and 2. Lower bound.}

In \cite{Sepp, Sepp2} the lower bound was obtained with a convexity argument on the scaling limit.
Instead, we adapt a more probabilist argument due to \rvsv (proof of Theorem 4.1 in \cite{SidoBroken2}) for
the closely related model of last passage percolation with geometric weights.
Here, we only give the details for Model 1 but the same argument applies for Model 2
(there are only minor modifications due mainly to the fact that sinks have no more a Bernoulli distribution but a geometric distribution).

Let us consider Model 1 on the rectangle $[0,an]\times [0,bn]$ with sinks and sources with optimal source intensity $\alpha=\alpha(a,b)$ and $\alpha^\star=\alpha^\star(a,b)$ (see \eqref{Eq:OptiRectangle}).
For $\eps \in  [0,1]$, denote by $L_{(an,bn)}^{(1)}(\eps)$ the length of the largest non-decreasing subsequence defined by
\begin{multline*}
L_{(an,bn)}^{(1)}(\eps)=\max\left\{L; (i_1,0),\ldots, (i_k,0),(i_{k+1},j_{k+1}),\dots,(i_L,j_L)\in \xi,\right.\\
\left. 0<i_1< \dots<i_k  \le an\eps < i_{k+1}<\ldots< i_L \le an \text{ and } 0<j_{k+1}<\dots <j_L \le bn\right\}.
\end{multline*}
Here is an example where $L_{(an,bn)}^{(1)}(\eps)=4$.
Note that because of the constraint $L_{(an,bn)}^{(1)}(\eps)$ is smaller than the number of Hammersley's lines 
$\mathcal{L}_{(an,bn)}^{(1),\alpha,\alpha^\star}$.
\begin{center}
\includegraphics[width=6cm]{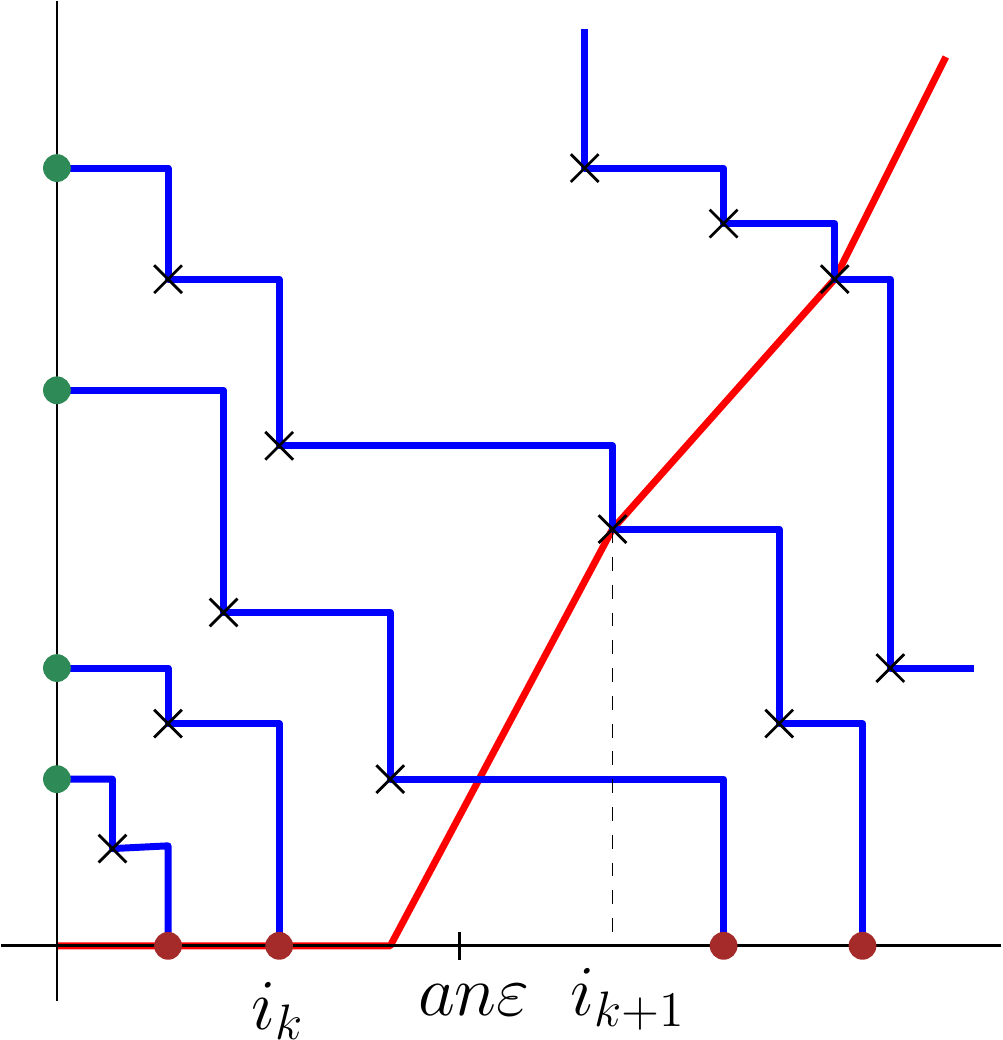}
\end{center}

Recall that $\phi^{(1)}$ has been defined in \eqref{eq:tgv}.
\begin{lem}[Largest subsequence using some sources]\label{Lem:OptiBord}
There exist positive and non-decreasing 
functions $f$, $g$ on $(0,1]$, which depend on $a,b$, such that
$$\mathbb{P}\left(L_{(an,bn)}^{(1)}(\eps) \ge n\left(\phi^{(1)}(a,b,\alpha(a,b))-f(\eps)\right)\right)\le \exp(-ng(\eps)).$$
\end{lem}

\begin{proof}[Proof of Lemma \ref{Lem:OptiBord}]
We write $L_{(an,bn)}^{(1)}(\eps)=I_1+I_2$ with
\begin{eqnarray*}I_1&:=&\max\left\{k; (i_1,0),\ldots, (i_k,0)\in \xi, i_1< \dots<i_k \le an\eps\right\}\\
I_2&:=&\max\left\{k; (i_1,j_1),\ldots, (i_k,j_{k})\in \xi,an\eps < i_1< \dots<i_k \le an\text{ and } 0<j_{1}<\dots <j_k \le bn\right\}.
\end{eqnarray*}
The random variables $I_1$ and $I_2$ are independent,  $I_1$ is binomially distributed with parameters $(\lfloor an\eps \rfloor, \alpha(a,b))$.
The random variable $I_2$ has the law of $L^{(1)}_{(an(1-\eps),bn)}$, thus is dominated by a random variable $I_3$ with law $\mathcal{L}_{(an(1-\eps),bn)}^{(1),\alpha(a(1-\eps),b),\alpha^\star(a(1-\eps),b)}$.
Using \eqref{eq:tgv}, we get:
\begin{eqnarray*}
\frac{1}{n}\mathbb{E}\left[L_{(an,bn)}^{(1)}(\eps) \right]
 & \le & \frac{1}{n}\mathbb{E}\left[I_1+I_3 \right] \\
  & \le & a\eps\alpha(a,b) + \phi^{(1)}\Big(a(1-\eps),b,\alpha(a(1-\eps),b)\Big),
\end{eqnarray*}
where, by convention, $\alpha(a',b')=1$ when $p \ge a'/b'$ and $\alpha(a',b')=0$ when $p \ge b'/a'$.
One can check that for $p<\min\{a/b,b/a\}$ and $a',b'>0$  such that $a/b \neq  a'/b'$ we have
$$
\phi^{(1)}(a',b',\alpha(a,b)) > \phi^{(1)}(a',b',\alpha(a',b')).
$$
This is due to the fact that $\alpha(a,b) \neq \alpha(a',b')$ and that there is a unique $\alpha \in [0,1]$ which minimizes $\phi(a',b',\alpha)$. Thus,
\begin{eqnarray*}
a\eps\alpha(a,b) + \phi^{(1)}\Big(a(1-\eps),b,\alpha(a(1-\eps),b)\Big)&<&a\eps\alpha(a,b) + \phi^{(1)}\Big(a(1-\eps),b,\alpha(a,b)\Big)\\
&= & a\eps\alpha(a,b) + a(1-\eps)\alpha(a,b) + b\alpha^\star(a,b)  \\
 & = & a\alpha(a,b) + b\alpha^\star(a,b) \\
 & = & \phi^{(1)}(a,b,\alpha(a,b)).
\end{eqnarray*}
Denote by $d(\eps)$ the difference between the right and left hand side of the inequality. The function $d$ is positive on $(0,1]$ and continuous. We set 
$f(\eps)=\tfrac 1 2 \min\{d(\delta), \delta\in [\eps,1]\}$ which is non-decreasing and positive. Then, we have
\begin{eqnarray*}
\mathbb{P}\left(L_{(an,bn)}^{(1)}(\eps) \ge n\left(\phi^{(1)}(a,b,\alpha(a,b))-f(\eps)\right)\right) 
 &\le &  \mathbb{P}\left( I_1+I_3 \ge \mathbb{E}(I_1+I_3) + n  \big( d(\eps)-f(\eps) \big) \right)\\
 &\le & \mathbb{P}\left( I_1+I_3 \ge \mathbb{E}(I_1+I_3) + n  f(\eps) \right).
\end{eqnarray*}
Recalling that $I_1$ is  binomially distributed  and $I_3$ is the sum of two binomial random variables, 
function $g$ is  obtained by applying Hoeffding's 
 inequality to  $I_1$ and $I_3$.

\end{proof}

We still consider  Model 1 on the rectangle $[0,an]\times [0,bn]$ with sinks and sources with optimal source intensity $\alpha=\alpha(a,b)$ and $\alpha^\star=\alpha^\star(a,b)$. 
Let $\pi_n$ be an optimal path for $\mathcal{L}_{(an,bn)}^{(1),\alpha,\alpha^\star}$.
If there are several optimal paths, we choose one of them in an arbitrary way.
Define $D_n$ as the number of sources $(i,j)\in \pi_n$ with $j=0$.

\begin{lem}[Optimal paths do not take many sources]\label{Lem:MajoPointsDuBord} There exists a positive function $h$ on $(0,1]$ such that for 
and any $\delta \in (0,1]$ and for any $n$ large enough 
$$
\mathbb{P}\left( D_n>an\delta\right)\le \exp(-nh(\delta)).
$$
In particular $D_n/n$ converges to 0 in $L^1$.
\end{lem}
\begin{proof}[Proof of Lemma \ref{Lem:MajoPointsDuBord}]
By definition, if $D_n>an\delta$ then for some $\eps \ge \delta$ such that $\eps a n \in \N$ and $\epsilon \le 1$
\begin{itemize}
\item there are more than $an\delta$ sources in $\set{1,\dots,\eps an}$;
\item $L_{(an,bn)}^{(1)}(\eps)= \mathcal{L}_{(an,bn)}^{(1),\alpha,\alpha^\star}=\mathcal{T}_{(an,bn)}^{(1),\alpha,\alpha^\star}+\#\set{\text{sinks between $1$ and $bn$}}$.
\end{itemize}
Therefore
\begin{align*}
\mathbb{P}\left\{ D_n>an\delta\right\}
&\leq \mathbb{P}\left( \mathcal{T}_{(an,bn)}^{(1),\alpha,\alpha^\star} \le n(a\alpha-\frac{1}{2}f(\delta))\right)\\
&+  \mathbb{P}\left( \#\set{\text{sinks}} \le n(b\alpha^\star-\frac{1}{2}f(\delta))\right)\\
&+ \mathbb{P}\left(L_{(an,bn)}^{(1)}(\eps)  > n\left(a\alpha+b\alpha^\star-f(\delta)\right) \text{for some }\eps\text{ as above}\right)\\
&\leq \exp(-n\tilde{g}(\delta))
+\sum_{1 \ge \eps \ge \delta;\ \eps an\in \N} \mathbb{P}\left(L_{(an,bn)}^{(1)}(\eps)  > n\left(\phi^{(1)}(a,b,\alpha(a,b))-f(\delta)\right) \right)\\
&\leq \exp(-n\tilde{g}(\delta))
+\sum_{1 \ge \eps \ge \delta;\ \eps an\in \N} \mathbb{P}\left(L_{(an,bn)}^{(1)}(\eps)  > n\left(\phi^{(1)}(a,b,\alpha(a,b))-f(\eps)\right) \right)\\
&\leq \exp(-n\tilde{g}(\delta))
+\sum_{1 \ge \eps \ge \delta;\ \eps an\in \N}  \exp(-ng(\eps))  \\
 & \leq  \exp(-n\tilde{g}(\delta))+an\exp(-ng(\delta))
\end{align*}
for large $n$, where we used Lemma \ref{Lem:OptiBord} and where $\tilde{g}$ is some positive function.
This implies the first part of the lemma.
The convergence of $D_n/n$ to 0 in probability and in $L^1$ follows, since the sequence is bounded. 
\end{proof}

We now conclude the proof of the lower bound noticing that we have 
\begin{equation*}
\mathcal{L}_{(an,nb)}^{(1),\alpha,\alpha^\star}-D_n-D_n' \le L_{(an,bn)}^{(1)}
\end{equation*}
where $D'_n$ is the number of sinks $(i,j)\in \pi_n$ with $i=0$.
By the previous lemma we know that $D_n/n$ tends to $0$ in $L^1$.
By symmetry, the same result holds for $D'_n/n$.
Taking expectations of both sides in the previous inequality we get
$$
\liminf_{n\to +\infty} \frac{1}{n}\mathbb{E}[L_{(an,bn)}^{(1)}] \geq a\alpha +b\alpha^\star=\phi^{(1)}(a,b,\alpha(a,b)).
$$
\end{proof}
\subsection*{Back to  Ulam's constant}
We observe that if we take $p=1/n$, Theorem 1 for Model 1 suggests $L_n^{(1)}\approx 2\sqrt{n}$, which is consistent with the asymptotics of  Ulam's problem, since $\xi$ is then close, after renormalization, to a Poisson point process with intensity $n$.

In fact, one can rigorously recover that $c=2$ using our proof of Theorem 1. To do so,  consider a  Poisson point process $\Xi$ with intensity $n$ in the unit square and denote by $\ell(n)$ the greatest number of points of $\Xi$ an increasing path can go through. To get   a lower bound and an upper bound of $\ell(n)$, we divide the square $[0,1]^2$ in two different ways.

 First, we fix some $k\ge 1$ and  we divide it into small squares of length side $1/(k\sqrt{n})$. Say that $\xi_{j}(i)=1$ if at least one point of $\Xi$ is in the square with top-right corner $(i/(k\sqrt{n)},j/(k\sqrt{n}))$ and consider the quantity 
$L_{k\sqrt{n}}^{(1)}:=L_{(k\sqrt{n},k\sqrt n)}^{(1)}$ associated to the family $(\xi_{j}(i))_{i,j \le k\sqrt{n}}$. It is clear that 
 \begin{equation}\label{eq:ulam1}
 \ell(n)\ge L_{k\sqrt{n}}^{(1)}.
 \end{equation} Denoting 
\begin{equation*}p_k=
\mathbb{P}(\xi_{j}(i)=1)=\mathbb{P}\left(\text{Poiss}(1/k^2) \geq 1 \right)=1-e^{-1/k^2},
\end{equation*}
Theorem \ref{TheoLLN} implies 
\begin{equation*}
\frac{L_{k\sqrt{n}}^{(1)}}{\sqrt{n}}\xrightarrow[n\to +\infty]{\text{a.s.}} \frac{2k\sqrt{p_k}}{\sqrt{p_k}+1}.
\end{equation*}
 Using \eqref{eq:ulam1} and letting $k$ tend to infinity, we get $c:=\lim \ell(n)/\sqrt{n} \ge 2$.
 
To prove the upper bound, we divide now the square into small squares of length side $1/n^4$. Say that $\xi_{j}(i)=1$ if at least one point of $\Xi$ is in the square with top-right corner $(i/n^4,j/n^4)$ and consider the quantity $L_{n^4}^{(1)}:=L_{(n^4,n^4)}^{(1)}$ associated to the family $(\xi_{j}(i))_{i,j \le n^4}$. The parameter of these Bernoulli random variables is now 
$$
\tilde{p}_n=\mathbb{P}\left(\text{Poiss}(1/n^7) \geq 1 \right)=1-e^{-1/n^7}.
$$
With probability higher than $1-n^{-2}$,  all the columns and lines of width $1/n^4$ contain at most one point of $\Xi$. On this event that we denote $F_n$, $L_{n^4}^{(1)}$ coincides with $\ell(n)$.
We now use some intermediate results of the proof of the upper bound for $a=b=1$.
Inequality \eqref{Eq:DominationPuits1} still holds despite the dependence of $\tilde{p}_n$ on $n$ \emph{i.e.}
$$
 L_{n^4}^{(1)}\leq \mathcal{L}_{(n^4,n^4)}^{(1),\alpha_n,\alpha_n^\star},
$$
with $\alpha_n=\alpha_n^\star=\frac{\sqrt{p_n}-p_n}{1-p_n}\sim n^{-7/2}$. Using \eqref{eq:tgv},we get
\begin{align*}
\mathbb{E}[\ell(n)]&\leq \mathbb{E}[|\Xi| \mathbf{1}_{\bar{F_n}}] +
\mathbb{E}[\mathcal{L}_{(n^4,n^4)}^{(1),\alpha_n,\alpha_n^\star}] \\
&\leq \sqrt{\mathbb{E}[\mathrm{Poiss}(n)^2] /n^2} +n^4(\alpha_n+\alpha_n^\star),
\end{align*}
where we used the Cauchy-Schwarz inequality. Dividing by $n$ and letting $n$ go to infinity, we obtain $c\leq 2$.


\section{Stationary measures of Hammersley's process on $\Z$}\label{Sec:SurZ}

In this section, which can be read independently from the rest of the article, we study the analogous on the infinite line $\bbZ$ of the process $(X^1_{t})_{t\geq 0}$ defined in \eqref{eq:defX}. This infinite Hammersley process was introduced by T.Sepp\"al\"ainen to prove the Law of Large Numbers for Problem 1.

In order to generalize $(X_{t})_{t\geq 0}$ to a process taking its values in $\{0,1\}^\bbZ$ we need a construction that does not rely on Hammersley lines.
As before, $\left(\xi_{t}(i)\right)_{i\in\bbZ,t\in \mathbb{N}}$ denote i.i.d. Bernoulli$(p)$ random variables we say that there is a \emph{cross} at time $t$ located at $x$ if $\xi_t(x)=1$.
Informally, the infinite Hammersley process is defined as follows:
\begin{center}\emph{
At time $t$, if there is a particle at $x$ and if the particle immediately on its left is at $y$, then the particle at $x$ jumps at time $t+1$ at the rightmost cross of $\xi_{t+1}$ in interval $(y,x)$ (if any, otherwise it stays still).}
\end{center}

For our purpose, we need a construction of $(X_t)$ in which we let the crosses \emph{act} one by one on configurations.
We first draw $X_0\in\{0,1\}^\bbZ$ at random according to some distribution  $\mu$ such that $\mu-$almost surely, for all $i$ there is $j<i$ such that $X_0(j)=1$.

We now explain how to construct $(X_{t+1})$ from $(X_t)$. A cross will \emph{act} on a given configuration of particles in the following way: if a cross is located at $x$, then the leftmost particle of $(X_t)$ in the interval $\llbracket x,+\infty \llbracket$ (if any) moves to $x$; if there is no such particle then a particle is created at $x$. 

With this definition, we can now construct the value of $(X_{t+1})$ as a function of $(X_t)$ and the crosses at time $t+1$:
We define $X_{t+1}$ as the result of the {\em successive} actions on $X_t$ of all crosses at time $t+1$ {\em from the right to the left}.
Note that by definition, there only a finite number of crosses in $\xi_{t+1}$ that can modify $X_t(i)$: those between $X_t(i)$ and the first particle on its left.
 
An example is  drawn in this figure where the circles represent the particles and the crosses  the locations where $\xi_{t+1}$ equals $1$. Here, we have $X^1_{t+1}=(1,0,0,1,0,0,0,1,1,0,0,...)$ (as in the previous sections, time goes from bottom to top in our pictures):

\begin{center}
\medskip
\includegraphics[width=11cm]{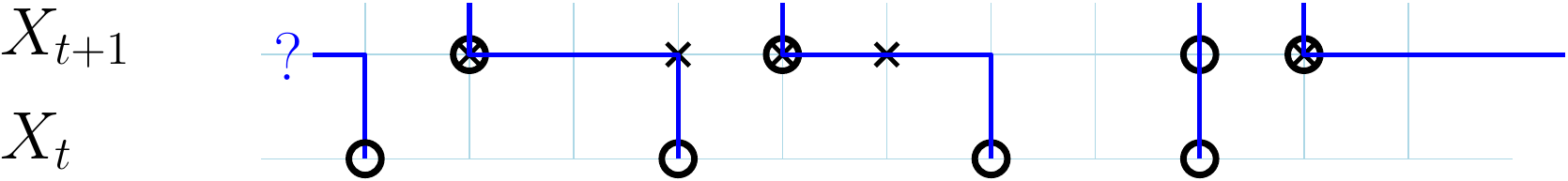}
\medskip
\end{center}

\begin{theo}\label{Theo:Extremality}
The only extremal translation-invariant stationary measures of  Hammersley's process on $\bbZ$ are measures $\mathrm{Ber}(\alpha)^{\otimes \Z}$ for all $\alpha\in (0,1]$.
\end{theo}

This theorem is precisely the discrete counterpart of (\cite{AldousDiaconis}, Lemma 7).
The fact that Bernoulli product measures are stationary was proved by T.Sepp\"al\"ainen (\cite{Sepp} Lemma 2.1). 
The reason for which we only focus on Model $1$ on the infinite line is that the analogous of Model $2$ is much more complicated to analyze. Indeed, the evolution of a particle in Model $2$ depends on the whole configuration on its left. We don't know if a similar statement to Theorem \ref{Theo:Extremality} holds for Model 2.

In order to prove that there are no other extremal measures we adapt the proof of Jockush-Propp-Shor (\cite{ArcticCircle} p.17-20) for the discrete-time TASEP (see also \cite{Liggett} Ch.VIII). The main difference here with Jockush-Propp-Shor is our Lemma \ref{l:new} below that replaces their \emph{block argument}.

\begin{proof}[Proof of extremality]
Let $\alpha\in (0,1]$ and $\mu$ be an extremal translation-invariant stationary measure on $\set{0,1}^\Z$ with marginals
$\mu(X(i)=1)=\alpha$, we want to prove that $\mu=\mathrm{Ber}(\alpha)^{\otimes \Z}$.

The proof is made of the following steps: 
\begin{enumerate}
\item we introduce a measure $\underline{\pi}$ on $\set{0,1}^\Z\times \set{0,1}^\Z$ which is a \emph{minimal} coupling between  $\mathrm{Ber}(\alpha)^{\otimes \Z}$  and $\mu$;
\item we prove (Lemma \ref{Lem:MotifInterdit}) that some patterns that would decrease $\card\set{i;x_i\neq y_i}$, have $\underline{\pi}$-probability zero;
\item we conclude: $\mu=\mathrm{Ber}(\alpha)^{\otimes \Z}$.
\end{enumerate}

Let $\mathcal{M}_\star$ be the set of translation-invariant measures on $\set{0,1}^\Z\times \set{0,1}^\Z$ that are a coupling of $\mathrm{Ber}(\alpha)^{\otimes \Z}$ and $\mu$. The set $\mathcal{M}_\star$ is non-empty (it contains $\mathrm{Ber}(\alpha)^{\otimes \bbZ}\otimes \mu$) and compact for the weak topology.

We will prove that in $\mathcal{M}_\star$ there is a coupling $\left((X(i))_{i\in\Z},(Y(i))_{i\in\Z}\right)$ such that $X=Y$ a.s.
Set
$$
\begin{array}{r c c l}
D : & \mathcal{M}_\star & \to & [0,1]\\
& \pi         & \mapsto & \pi\left(X(0)\neq Y(0)\right).
\end{array}
$$
 The function $D$ is continuous on the compact set $\mathcal{M}_\star$ and thus attains its minimum $\delta \geq 0$ at some coupling $\underline{\pi}$. If $\delta=0$, then the theorem is proved: by translation-invariance $X=Y$ $\underline{\pi}$-a.s.

Let $(X_0,Y_0)\sim \underline{\pi}$, we use the same   random variables $\left(\xi_{t}(i)\right)_{i\in\bbZ,t\in \mathbb{N}}$ to define the dynamics of $X$ and $Y$.
We denote by $(X_t,Y_t)$ the pair of configurations at time $t$, its joint distribution is denoted by $\underline{\pi}_t$. By construction,  for every $t$, $\underline{\pi}_t$ is translation-invariant.
For every $k \ge 1$ we denote by
$$
\Delta_k(X_t,Y_t)
$$
the number of $i \in \{1,\dots,k\}$ such that $X_t(i) \neq Y_t(i)$.

Let $n \ge 2$ and $t \ge 0$.
We say that there is a $n-$forbidden pattern at location $x \in \Z$  at time $t$ if the configuration $(X_t,Y_t)$ between locations $x$ and $x+n-1$
is either
 $$
\begin{pmatrix}
        &                 & x   &   &         & x+n-1     & \\
\ \ X_t = &(&1   & 0 & \dots\ 0& 0     & )\ \ \\
\ \ Y_t = &(&0   & 0 & \dots\ 0& 1     & )\ \
\end{pmatrix}
$$
or
$$
\begin{pmatrix}
        &                 & x   &   &         & x+n-1     & \\
\ \ X_t = &(&0   & 0 & \dots\ 0& 1     & )\ \ \\
\ \ Y_t = &(&1   & 0 & \dots\ 0& 0     & )\ \
\end{pmatrix}.
$$

\begin{lem}\label{l:new} For $n \ge 2$, let $E_{x,n}$ be the event $\xi_1(x)=1$, $\xi_1(x+1)=\xi_1(x+2)=\dots =\xi_1(x+n-1)=0$.
Let $j \ge 1$, denote by $F(n,j)$ the subset of locations of $\{0,n,2n,\dots,(j-1)n\}$ at which is located a forbidden pattern at time $0$. Let 
$$A(n,j)=\card\{x\in F(n,j),\; E_{x,n} \mbox{ occurs}\}.$$
Then:
$$
\Delta_{jn}(X_1,Y_1) \le \Delta_{jn}(X_0,Y_0) + 1 - 2A(n,j).
$$

\end{lem}
\begin{proof}[Proof of Lemma \ref{l:new}]
We let the crosses at time $1$ act one by one from the right to the left.
We consider the impact of each cross action on the discrepancy $\Delta_{jn}$ between the two configurations. Let us note that,  when a cross located at $i$ acts on some configuration of particles, it  changes the value of at most two sites of the configuration: $i$ and the leftmost $1$ of the configuration in the interval $\llbracket i, \infty)$. Thus, we deduce the following facts:
\begin{itemize}
\item The crosses strictly to the right of $jn$ have no impact on the discrepancy.
\item A cross located at a point $x$ in $\{1,\dots,jn\}$ cannot increase the discrepancy.
To check this fact, one can study all the cases. They are all shown below, up to symmetries between $X$ and $Y$.
$$
\begin{pmatrix}
  &  x  & \\
 (&  1  & )\\
 (&  1  & )\\
\end{pmatrix},
\begin{pmatrix}
    & x &   &        &   &    & \\
 ( & 1  & ? & \cdots & ? & 1 & )\\
 (  & 0 & 0 & \cdots & 0 & 1 & )\\
\end{pmatrix},
$$
$$
\begin{pmatrix}
 & x &   &        &   &  & \\
  ( & 1 & ? & \cdots & ? & 0& )\\
  ( & 0 & 0 & \cdots & 0 & 1 & )\\
\end{pmatrix},
\begin{pmatrix}
&x &          &   & & \\
  ( &0 &  \cdots  & 0 & 1 &  )\\
  ( & 0 &  \cdots  & 0 & 1 & )\\
\end{pmatrix},
$$
$$
\begin{pmatrix}
 & x &        &   &   &   &   &   & & \\
 ( &0 & \cdots & 0 & 0 & 0 & 0 & 0 & 1 & )\\
 ( &0 & \cdots & 0 & 1 & ? & \cdots & ? & 1& )\\
\end{pmatrix},
\begin{pmatrix}
 & x &        &   &   &   &   &   & &\\
 (&0 & \cdots & 0 & 0 & 0 & 0 & 0 & 1 & )\\
 (&0 & \cdots & 0 & 1 & ? & \cdots & ? & 0& )\\
\end{pmatrix}.
$$
\item For a forbidden pattern at some $in$, if $E_{in,n}$ occurs, the cross at $in$ decreases the discrepancy by $2$.
\item In the worst case, the right-most cross which is strictly to the left of 1 increases the discrepancy by $1$.
\item The crosses to the left of the previous one have no impact on the discrepancy.
\end{itemize}
The result follows.
\end{proof}

\begin{lem}[Forbidden patterns]\label{Lem:MotifInterdit} 
Let $n \ge 2$.
The $\underline{\pi}_0$ probability of a $n$-forbidden pattern at any location is $0$.
\end{lem}
\begin{proof}[Proof of Lemma \ref{Lem:MotifInterdit}]
Let $j \ge 1$.
By Lemma \ref{l:new} we have:
$$
\Delta_{jn}(X_1,Y_1) \le \Delta_{jn}(X_0,Y_0) + 1 - 2A(n,j).
$$
Note that the probability that, at some location $in$, there is a $n$-forbidden pattern and that $E_{in,n}$ occurs is $\zeta p(1-p)^{n-1}$
where $\zeta$ is the probability of a $n$-forbidden pattern at a given location.
We aim to prove $\zeta=0$.
Taking expectation in the previous display we get
$$
jn D(\underline{\pi}_1) \le jn D(\underline{\pi}_0) + 1 - 2j\zeta p(1-p)^{n-1}.
$$
Dividing by $nj$ and letting $j \to \infty$ we get:
$$
D(\underline{\pi}_1) \le D(\underline{\pi}_0) - \frac{2\zeta p(1-p)^{n-1}}n.
$$
But the minimality of $\underline{\pi}$ yields $D(\underline{\pi}_1) \ge D(\underline{\pi}_0)$ and we get $\zeta=0$.
\end{proof}

\begin{lem}[There is a forbidden pattern somewhere]
Assume that $\delta>0$, there exists $n$ such that\emph{
$$
\underline{\pi}
\begin{pmatrix}
           &               & 1   &   &         & n     & \\
\ \ X = &(&1   & \text{?} & \dots\ \text{?}& 0     & )\ \ \\
\ \ Y = &(&0   & \text{?} & \dots\ \text{?}& 1     & )\ \
\end{pmatrix}
>0
\quad \text{ or }\
\underline{\pi}
\begin{pmatrix}
           &               & 1   &   &         & n     & \\
\ \ X = &(&0   & \text{?} & \dots\ \text{?}& 1     & )\ \ \\
\ \ Y = &(&1   & \text{?} & \dots\ \text{?}& 0     & )\ \
\end{pmatrix}
>0
$$}
\end{lem}
\begin{proof}
Let $(X,Y)\sim \underline{\pi}$,
$$
\delta = \underline{\pi}(X(0)\neq Y(0))=  \underline{\pi}(X(0)=1,Y(0)=0)+ \underline{\pi}(X(0)=0, Y(0)=1).
$$
The two terms on the right side are equal:
\begin{align*}
\underline{\pi}(X(0)=1,Y(0)=0)&=\underline{\pi}(X(0)=1)-\underline{\pi}(X(0)=1,Y(0)=1)\\
&= \underline{\pi}(Y(0)=1)-\underline{\pi}(X(0)=1,Y(0)=1)\quad (X,Y\text{ have same marginals})\\
&= \underline{\pi}(X(0)=0,Y(0)=1).
\end{align*}
Then $0<\delta/2=\underline{\pi}(X(0)=1,Y(0)=0)$. Assume that the lemma is false (for every $n$), we have  $\underline{\pi}(A\cup B)=1$, where
\begin{align*}
A&=\set{(x,y),x_i\leq y_i\text{ for all }i\in\bbZ},\\
B&=\set{(x,y),x_i\geq y_i\text{ for all }i\in\bbZ}.
\end{align*}
As seen above, $\delta>0$ implies then that $\underline{\pi}(A),\underline{\pi}(B \backslash A)>0$.
Besides, it is easy to check that $A$ and $B\backslash A$ are preserved by the dynamics. Hence, considering the restriction of $\underline{\pi}$ on these two subsets, $\underline{\pi}_{A}:=(\nu_A,\mu_A)$ and $\underline{\pi}_{B \backslash A}:=(\nu_{B \backslash A},\mu_{B \backslash A})$, we see that $\nu_A$ (resp. $\mu_A$) and $\nu_{B \backslash A}$ (resp. $\mu_{B \backslash A}$ ) define two  translation-invariant stationary measures such that
\begin{eqnarray*}
\mathrm{Ber}(\alpha)^{\otimes \Z}&=&\underline{\pi}(A)\nu_A+ \underline{\pi}(B \backslash A)\nu_{B \backslash A}.\\
\mu&=&\underline{\pi}(A)\mu_A+ \underline{\pi}(B \backslash A)\mu_{B \backslash A}.
\end{eqnarray*}
The extremality of $\mathrm{Ber}(\alpha)^{\otimes \Z}$ and $\mu$ implies that $\mathrm{Ber}(\alpha)^{\otimes \Z}=\nu_A=\nu_{B \backslash A}$ and $\mu=\mu_A=\mu_{B \backslash A}$. But, by definition of $A$ and $B$, $\nu_{A} \preccurlyeq \mu_A$ and $\mu_{B \backslash A} \preccurlyeq \nu_{B \backslash A}$. Thus, necessarily, $\mu=\mathrm{Ber}(\alpha)^{\otimes \Z}$.
\end{proof}
Now we can obtain our contradiction. Indeed, there exists an integer $n$ such that
$$
\underline{\pi}
\begin{pmatrix}
           &               & 1   &   &         & n     & \\
\ \ X = &(&1   & \text{?} & \dots\ \text{?}& 0     & )\ \ \\
\ \ Y = &(&0   & \text{?} & \dots\ \text{?}& 1     & )\ \
\end{pmatrix}
>0
\quad \text{ or }\
\underline{\pi}
\begin{pmatrix}
           &               & 1   &   &         & n     & \\
\ \ X = &(&0   & \text{?} & \dots\ \text{?}& 1     & )\ \ \\
\ \ Y = &(&1   & \text{?} & \dots\ \text{?}& 0     & )\ \
\end{pmatrix}
>0.
$$
Without loss of generality we do the first case. We can assume that $X,Y$ coincide between $2$ and $n-1$, otherwise we can decrease $n$.
Let $1 < k <n$ be the leftmost position at which $\binom{{X}({k})}{Y({k})}= \binom{1}{1}$.

With positive probability we can turn this $\binom{1}{1}$ into $\binom{1}{0}$:
$$
\begin{pmatrix}
        &   &\times & - &   -  & -  & - &      -     & - &\\
        &   &   1   &   &       &    & k &           & n &\\
\ \ X_t = & ( & 1     & 0 & \dots & 0  & 1 & \text{?}  & 0 &)\ \ \\
\ \ Y_t = & ( & 0     & 0 & \dots & 0  & 1 & \text{?}  & 1 &)\ \
\end{pmatrix}
\stackrel{t+1}{\to}
\begin{pmatrix}
        &   &  &   &      &    &   &            &   &\\
        &   &   1   &   &       &    &  k &           & n &\\
\ \ X_{t+1} = & ( & \text{?} & 0 & \dots & 0  & 1 & \text{?}  & 0 &)\ \ \\
\ \ Y_{t+1} = & ( & \text{?} & 0 & \dots & 0  & 0 & \text{?}  & 1 &)\ \
\end{pmatrix}
$$
We repeat this process between positions $k$ and $n$ and finally for some $t,i,j$ we have
$$
\underline{\pi}
\begin{pmatrix}
             &             & i   &   &         & j     & \\
\ \ X_t = &(&1   & 0 & \dots\ 0& 0     & )\ \ \\
\ \ Y_t = &(&0   & 0 & \dots\ 0& 1     & )\ \
\end{pmatrix}
>0,
$$
which contradicts Lemma \ref{Lem:MotifInterdit}.
\end{proof}

\appendix

\section{Dynamic of the underlying interacting particle system} \label{s:leretour}

In this appendix we explicit the dynamic of the underlying interacting particle system in Models 1 and 2.
Consider the processes $(X^i_t)_{t\ge 0}$ defined by \eqref{eq:defX} (without sources and sinks). 
One can explicitely construct  the dynamic of these processes in term of the family of 
i.i.d. random variables $\{\xi_{t}(x), x\in \llbracket 1,n \rrbracket, t\in \N\}$   with law Bernoulli$(p)$ (see \eqref{e:tram}). 
To do so, let say that there is a \emph{cross}  at time $t$ located at  $x$ if $\xi_t(x)=1$. 
A cross will \emph{act} on a given configuration of particles in the following way : if a cross is located at $x$, then
the left-most particle in the interval $\llbracket x,n \rrbracket$ (if any) moves to $x$ ; if there is no such particle then a particle is created at $x$. 

With this definition, we can now construct  the value of $X_{t+1}^i$ as a function of $X_t^i$ and the  crosses at time $t+1$:

\vspace*{0.2cm}

\noindent {\bf Model 1.} 
We define $X^1_{t+1}$ as the result of the {\em successive} actions on $X^1_t$ of all crosses at time $t+1$ {\em from the right to the left}.

An example is  drawn in this figure where the circles represent the particles and the crosses  the locations where $\xi_{t+1}$ equals $1$. Here, we have $X^1_{t+1}=(1,0,0,1,0,0,0,1,1,0,0)$ (as it is usual in the literature on Hammersley's processes, time goes  from bottom to top in our pictures):

\begin{center}
\includegraphics[width=11cm]{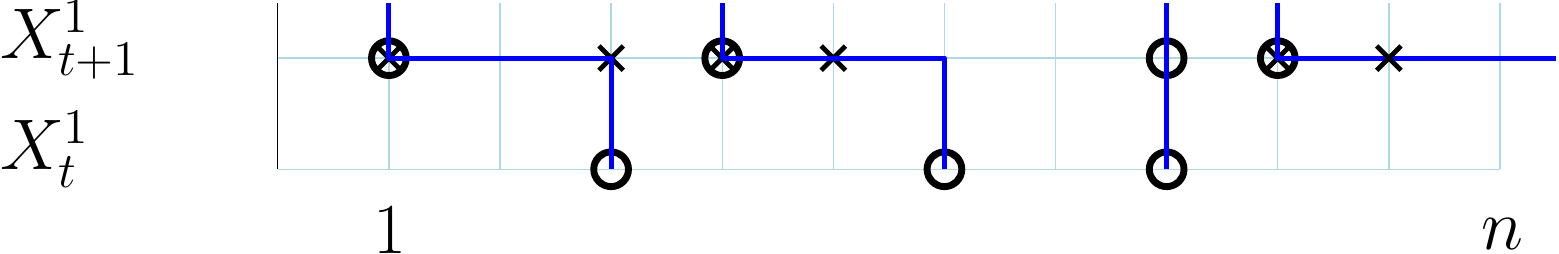}
\end{center}

Let us note that there is an alternative way to construct the model 1 :
At time $t$, if there is a particle at $x$ and if the particle immediately on its left is at $y$, then the particle at $x$ jumps at time $t+1$ at the rightmost cross of $\xi_{t+1}$ in interval $(y,x)$ (if any, otherwise it stays still). 
If there are some points  of $\xi_{t+1}$ at the right of the rightmost particle of $X^1_t$ then a new particle appears at the leftmost point among them.

\vspace*{0.2cm}

\noindent {\bf Model 2.} We define $X^2_{t+1}$ as the result of the {\em successive} actions on $X^2_t$ of all crosses at time $t+1$ {\em from the left to the right}.

\begin{center}
\includegraphics[width=11cm]{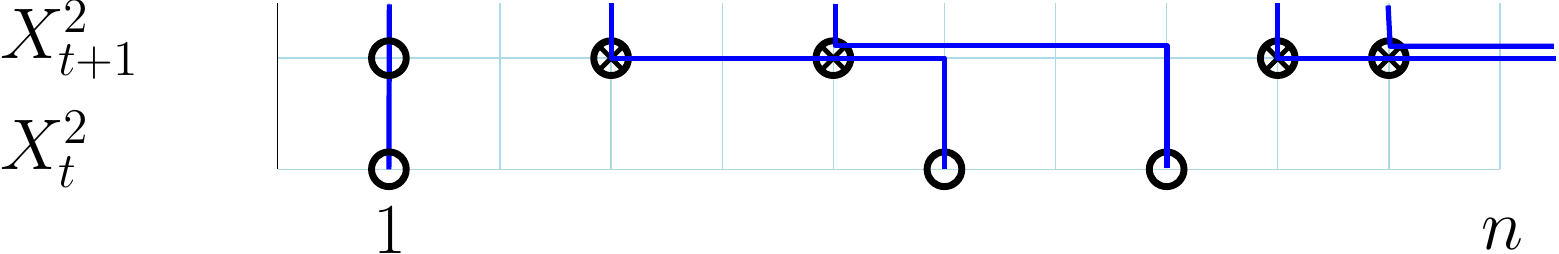}
\end{center}

We let the reader convince themselves that this construction coincides with the definition given in \eqref{eq:defX}.

Although the definition of both models seems, at first glance, very close, the nature of the two processes is in fact quite different. Indeed, in the first model, to find the location at time $t+1$ of a particle located at time $t$ at $x$, one just need to know the location $y$ of the particle immediately on his left at time $t$ and the position of the crosses of $\xi_{t+1}$ in the interval $(y,x)$. In particular, with no difficulty, one can define a  process with similar transitions on the whole line $\Z$ as it is done in \cite{Sepp}.
 For Model 2, the dependances are more intricate. Indeed, to determine the location at time $t+1$ of a particle located at time $t$ at $x$, one  need to know  the whole configuration of $X_{t}^2$ and $\xi_{t+1}$ on the interval $\llbracket 1,x \rrbracket$. In particular, the definition on the whole line $\Z$ of a similar process is more delicate and requires a condition between the density of crosses and particles (see \cite{Sepp2}).

 \medskip
 
Finally, let us note that the construction can be extended to the models with sources and sinks. 
The set of sources simply give the initial configuration of the particle system.
The sinks acts as the crosses.
In Model 1 sinks act after the other crosses while in Model 2 sinks act before the other crosses.

\paragraph{Aknowledgements.} The authors are glad to aknowledge David Aldous who was at the origin of their interest in this problem.

\end{document}